\documentclass{amsart}
\pagespan{1}{\pageref*{LastPage}}
\usepackage{etoolbox,lastpage}
\commby{}
\date{}
\usepackage{amsmath,amsthm,amscd,amsfonts,amssymb,enumerate}
\usepackage{graphicx}		
\usepackage{color}
\usepackage[colorlinks]{hyperref}

\newcommand{\N}{\mathbf N}

\newcommand{\skp}[2]{\left<#1,#2\right>}

\renewcommand{\a}{\alpha}

\newcommand{\mA}{\mathcal A}

\newcommand{\va}{\varphi}

\newtheorem{theorem}{Theorem}[section]
\newtheorem{proposition}[theorem]{Proposition}
\newtheorem{lemma}[theorem]{Lemma}
\newtheorem{corollary}[theorem]{Corollary}
\theoremstyle{definition}
\newtheorem{definition}[theorem]{Definition}
\newtheorem{example}[theorem]{Example}
\theoremstyle{remark}
\newtheorem{remark}[theorem]{Remark}
\numberwithin{equation}{section}

\title[Measures of noncompactness in Hilbert $C^*$-modules]{Measures of noncompactness in Hilbert $C^*$-modules}

\author[D. Ke\v{c}ki\'c]{Dragoljub J.\ Ke\v{c}ki\'c}
\address{Faculty of Mathematics University of Belgrade 11000 Belgrade  Serbia}
\email{keckic@matf.bg.ac.rs}

\author[Z.\ Lazovi\'c]{Zlatko Lazovi\'c}
\address{Faculty of Mathematics University of Belgrade 11000 Belgrade  Serbia}
\email{zlatkol@matf.bg.ac.rs}

\begin{document}

\begin{abstract}
Consider a countably generated Hilbert $C^*$-module $\mathcal M$ over a $C^*$-algebra $\mA$. There is a measure of noncompactness $\lambda$ defined, roughly as the distance from finitely generated projective submodules, which is independent of any topology. We compare $\lambda$ to the Hausdorff measure of noncompactness with respect to the family of seminorms that induce a topology recently iontroduced by Troitsky, denoted by $\chi^*$. We obtain $\lambda\equiv\chi^*$. Related inequalities involving other known measures of noncompactness, e.g.\ Kuratowski and Istr\u{a}\c{t}escu are laso obtained as well as some related results on adjontable operators.
\end{abstract}

\maketitle

\section{Introduction}

Hilbert modules, i.e.\ modules over some $C^*$-algebra have been studied intensively for several decades. Briefly, they are counterparts of Hilbert spaces where the field $\mathbb C$ of complex scalars is replaced by some $C^*$-algebra.

\begin{definition} Let $\mA$ be some $C^*$-algebra. A pre-Hilbert $C^*$-module is a (right) $C^*$-module $\mathcal M$ equipped with a sesquilinear form: $\mathcal M\times\mathcal M\to\mA$  with the following properties:
\begin{itemize}
\item[i)] $\skp{x}{x}\geq0$ for each $x\in\mathcal M$;
\item[ii)] $\skp{x}{x}=0$ implies that $x = 0$;
\item[iii)] $\skp{x}{y}=\skp{y}{x}^*$ for each  $x,y\in\mathcal M$;
\item[iv)] $\skp{x}{ya}=\skp{x}{y}a$ for any $x,y\in\mathcal M$ and any $a\in\mA.$
\end{itemize}

A pre-Hilbert $C^*$-module over $\mA$ is a Hilbert $C^*$-module if it is complete with respect to the norm $\|x\|=\|\skp{x}{x}\|^{\frac{1}{2}}$.
\end{definition}

A Cauchy-Schwartz type inequality holds for Hilbert $C^*$-modules (\cite[Proposition 1.2.4]{Paschke})
$$\skp{x}{y}\skp{y}{x}\leq\|y\|^2\skp{x}{x}\quad\mbox{for each }x,y\in N.$$

\begin{example}\label{PrimerStandardni}

Given a unital $C^*$-algebra $\mathcal{A}$, the standard Hilbert module $H_{\mathcal{A}}$ is defined as
$$H_{\mathcal{A}}=\{x=(\xi_1,\xi_2,\ldots)\,\,|\,\,\xi_j\in\mathcal{A},\,\,\sum_{j=1}^{\infty}\xi_j^*\xi_j\,\,\mbox{converges in the norm topology}\}$$
with the $\mathcal{A}$-valued inner product
$$\skp{x}{y}=\sum_{j=1}^{\infty}\xi_j^*\eta_j,\quad x=(\xi_1,\xi_2,\ldots), y=(\eta_1,\eta_2,\ldots)\in H_{\mathcal{A}}.$$
\end{example}

Unlike Hilbert spaces, an $\mA$-linear bounded operator on a Hilbert module need not have an adjoint. Therefore, it is usual to consider $B^{a}(M)$ the set of all bounded, adjointable, $\mA$-linear operators on a Hilbert module $\mathcal M$.

Among them, $\mA$-linear combinations of
$$\Theta_{y,z}:\mathcal M\to\mathcal N, y\in \mathcal N,z\in\mathcal M,\qquad \Theta_{y,z}(x)=y\skp{z}{x}$$
are called finite rank operators. Those belonging to the norm closure of the set of finite rank operators are called "compact" or $\mA$-compact. In general, $\mA$-compact operators need not map bounded sets into relatively compact sets, as is the case in the framework of Hilbert (and also Banach) spaces, though they share many properties of proper compact operators on a Hilbert space (see \cite{Manuilov2}, \cite{Manuilov3}). Hence, "compact". Indeed, infinite dimensionality of the image of some bounded set is contained in the algebra of scalars.

For instance, $P_n:H_\mA\to H_\mA$ given by
\begin{equation}\label{PeeN}
P_n(\xi_1,\xi_2,\dots)=(\xi_1,\xi_2,\dots,\xi_n,0,0,\dots),
\end{equation}
is $\mA$-compact. Indeed, if $\mA$ is unital it can be written as $P_n=\sum_{k=1}^n\Theta_{e_k,e_k}$, where $e_k$ is the basic vector consisting of zeros except on $k$th place it has the unity of $\mA$. If $\mA$ is not ubnital, then $P_n$ can be obtained as the limit of $\sum_{k=1}^n\Theta_{e_k^\alpha,e_k^\alpha}$, where $e_k^\alpha$ has $e^\alpha$ the approximate identity on its $k$th coordinate. However the image of the unit ball in $H_\mA$ is the unit ball in $\mA^n$ which is obviously not compact, unless $\mA$ is finite dimensional itself.

For general literature concerning Hilbert modules over $C^*$-algebras, including the standard Hilbert module, the reader is referred to \cite{Lance} or \cite{Manuilov}.

In \cite{Keckic1}, the autors pose a question whether there exists a topology on Hilbert module $\mathcal N$,  where $\mA$ is a unital $C^*$-algebra, such that each operator on $M$ is $\mA$-compact iff it maps a unit ball (in the norm) to a totally bounded set, and gave a partial answer for the standard Hilbert module, $H_{\mA}$ (Example \ref{PrimerStandardni}), by constructing the appropriate topology, here denoted by $\tau_3$

Soon after, in \cite{Troi}, Troitsky gave a complete answer in the case of countable generated Hilbert $C^*$-modules, constructing more suitable topology, her denoted by $\tau$.

Both topologies were constructed via a family of seminorms, converting the underlying Hilbert module in a uniform sapce. This allows to consider measures of noncompactness that arise from some family of seminorms, in particular  Hausdorff, Kuratowski and Istr\u{a}\c{t}escu measure of noncompactness, see \cite{Marina} and \cite{Arandjelovic1}.

Despite of any uniform structure, there is a natural distance of a given set from finite rank sets, based on \cite[Proposition 2.6.2]{Manuilov} and the definition preceding it. It was introduced in \cite{Keckic2}, as follows:

\begin{definition}\label{DefLambda}
Let $E\subset H_{\mA}$ be a bounded set. The measure of noncompactness of $E$, denoted by $\lambda(E)$, is the greatest lower bound of all $\eta>0$ for which there exists a free finitely generated module $L\leq H_{\mA}$ such that
$$d(E,L):=\sup_{x\in E}\inf_{y\in L}\|x-y\|<\eta.$$
\end{definition}

In \cite{Keckic2} there were considered Hausdorff, Kuratowski and Istr\u{a}\c{t}escu measures of noncompactness, with respect to topology $\tau_3$ (or more precisely, the family of seminorms defining it), denoted by $\chi^*$, $\alpha^*$ and $I^*$, respectively. Besides the expected inequalities
\begin{equation}\label{UsualIneq}
\chi^*(E)\le I^*(E)\le\alpha^*(E)\le 2\chi^*(E),
\end{equation}
for any bounded subset $E$ of the standard Hilbert module $H_\mA$, the following relationships with the measure $\lambda$ was obtained. For an arbitrary bounded $E\subseteq H_\mA$ there holds
\begin{equation}\label{InequalitiesTau3}
    \chi^*(E)\le\lambda(E)\le\sqrt{\|E\|I^*(E)},\qquad\|E\|=\sup_{x\in E}\|x\|.
\end{equation}
The first inequality is proved for an arbitrary $C^*$-algebra $\mA$, whereas the second holds if $\mA=B(H)$.

The aim of this note is to examine an arbitrary countably generated Hilbert module $\mathcal M$ and the corresponding measures of noncompactness related to the topology $\tau$ defined by Troit\/sky in \cite{Troi}, i.e\ from the family of seminorms that define it, as well as their relationship to the measure $\lambda$.

In addition to the expected inequalities (\ref{UsualIneq}) we prove
$$\lambda(E)=\chi^*(E),$$
for all countably generated Hilbert modules over an arbitrary $C^*$-algebra, which is much stronger result then (\ref{InequalitiesTau3}). Some related results for adjointable operators are also given.

\section{Prerequisits}

In this section we list known results and definitions necessary for the main result.

\subsection{Measure of noncompactness on uniform spaces}

Uniform spaces are usually defined as spaces endowed with a family of sets in $X\times X$ given as some kind of neighborhoods of the diagonal, or so-called entourages \cite[p.\ 169]{Nikolas}, for our purposes it is more convenient to give an equivalent definition via a family of semimetrics.

\begin{definition} A nonempty set $X$ endowed with a family of pseudometrics $\{d_{\alpha}\}$ (functions $d_{\alpha}\colon X\times X\to[0,+\infty)$ satisfying $(i)\,\, d_{\alpha}(x,x)=0; (ii)\,\,$ $d_{\alpha}(x,y)=d_\alpha(y,x)$; $(iii)\,\,d_{\alpha}(x,y)\leq d_{\alpha}(x,z)+d_{\alpha}(z,y)$ for each $x,y,z\in X$) is called a {\it uniform space}.
\end{definition}

By \cite[Theorem 1, p.142]{Nikolas2} this definition is equivalent to the usual definition via entourages.

The family of sets $B_{d_{\alpha}}(x;\varepsilon)=\{y\in X\,\,|\,\,d_{\alpha}(x,y)<\varepsilon\}$ forms a basis for the corresponding topology.

\begin{definition}
    A set $E\subset X$ is {\it totally bounded} if for every $d_{\alpha}$ and every $\varepsilon>0$ there exists a finite collection $y_1,y_2,\ldots,y_n$ of elements of $E$ such that the sets $B_{\alpha}(y_j;\varepsilon)=\{x\in E\,\,|\,\,d_{\alpha}(x,y_j)<\varepsilon\},\, j=1,2,\ldots n$ form a cover of $E$.
\end{definition}

\begin{remark}
    In the usual framework, where uniform space is defined by entourages, there is also another different definition of totally bounded sets. These two definitions are equivalent, see \cite[page 198]{Kelly}
\end{remark}

It is well known that every locally convex topological vector space is a uniform space.

Three most common measures of noncompactness on metric spaces are Hausdorff measure $\chi$, Kuratowski measure $\alpha$ and Istr\u{a}\c{t}escu measure $I$. Nothing is lost if we replace metric by some pseudometric. Hence, for a given pseudometric $d$ we have
$$\chi_d(E)=\inf\{\varepsilon>0\mid E\subseteq\bigcup_{j=1}^m B_d(x_j,\varepsilon),\;\mbox{for some}\;x_j\in X\}.$$
$$\alpha_d(E)=\inf\{\varepsilon>0\mid E=\bigcup_{j=1}^{m}E_j,\;\mbox{for some}\;E_j,\;\mathrm{diam}(E_j)<\varepsilon\},$$
$$I_d(E)=\sup\{\varepsilon>0\mid\exists\mbox{ infinite}\;S\subset E\;\mbox{so that}\;\forall\,x_m\neq x_n\in S,d(x_n,x_m)\geq\varepsilon\,\}.$$

The following properties of $\chi$, $\alpha$ and $I$ are well known. Their proof can be found in \cite{Rakoc} and \cite{Arandjelovic2}.

\begin{proposition}\label{osobine MNC} Let $E$, $E_1$, $E_2$ be bounded subsets of some metric space $X$, and let $\mu$ denote either $\chi$ or $\alpha$ or $I$. Then
\begin{enumerate}
    \item if $E_1\subseteq E_2$ then $\mu(E_1)\le\mu(E_2)$;
    \item $\mu(E_1+ E_2)\leq \mu(E_1)+\mu(E_2)$, provided that $X$ is a Banach space;
    \item\label{KnownIneq} $\chi(E)\leq I(E)\leq\alpha(E)\leq2\chi(E)$.
\end{enumerate}
\end{proposition}

Thus, on a uniform space we have a family of measures of noncompactness, one for each pseudometric. Sadovskii \cite{Sadovski} considered them as functions that maps bounded (with respect to all pseudometrics) sets into functions which domain is the set of all pseudometrics. In other words he put $[\alpha(E)](d):=\alpha_d(E)$, $[\chi(E)](d)=\chi_d(E)$. A similar can be done for Istr\u{a}\c{t}escu measure $I$.

Instead of the family of measures of noncompactness, one for each pseudometric, it is possible to consider a single measure of noncompactness, as it was done in  \cite{Arandjelovic1} and \cite{Marina}, in the following way. To each pseudometric $d_i$ assign a function $\mu_i$ which satisfies some list of axioms ($\chi$, $\alpha$ and $I$ fulfill all of them). Such a function is called just measure of noncompactness, or shortly MNC. Then it is possible to take their supremum as the measure of noncompactness. Note, however, that such defined measure of noncompactness highly depend on the family of pseudometrics and can differ even if two families determine the same uniformity.

\begin{definition}
Let $X$ be a uniform space and let $\{d_i\,|\,i\in J\}$ be a family of pseudometrics which defines topology on
$X$. Denote by $\mu_i$ an arbitrary MNC with respect to the pseudometric space $(X,d_i)$ for each $i\in J$. For a bounded (w.r.t all pseudometrics) $E\subseteq X$ We define
$$\mu^*(E)=\sup_{i\in J} \mu_i(E).$$
\end{definition}

The following properties of such defined measure $\mu^*$ were proved.

\begin{proposition}\label{IvanAxiome}
Let $X$ be a uniform space with a family of pseudometrics $\{d_i\,|\,i\in J\}$. The function $\mu^*$  has the following properties
\begin{enumerate}
\item $\mu^*(E)=+\infty$ if and only if $E$ is unbounded;
\item $\mu^*(E)=\mu^*(\overline E)$;
\item\label{MNC*TotBound} from $\mu^*(E)=0$ it follows that $E$ is totally bounded set;
\item from $E\subseteq F$ it follows $\mu^*(E)\le\mu^*(F)$;
\item $\mu^*(E\cup\{x\})=\mu^*(E)$;
\item if $X$ is complete, and if $\{E_n\}_{n\in\N}$ is a sequence of closed subsets of $X$ such that
$E_{n+1}\subseteq E_n$ for each $n\in\N$ and $\lim_{n\to\infty}\mu^*(E_n)=0$, then $K=\bigcap_{n\in\N}E_n$ is
a nonempty compact set.
\end{enumerate}
\end{proposition}

\begin{remark}\label{DrugiSmer}
    Although only one direction of item (\ref{MNC*TotBound}) is proved in \cite{Arandjelovic1} and \cite{Marina}, the other is easy to see. Indeed, if $E$ is totally bounded then $\mu_i(E)=0$ for all $i$ and hence $\mu^*(E)=0$.
\end{remark}

\begin{lemma}\label{MNCinequalities}
    There holds
    $$\chi^*(E)\le I^*(E)\le\alpha^*(E)\le2\chi^*(E).$$
\end{lemma}

\begin{proof}
    It follows immediately from Proposition \ref{osobine MNC} (\ref{KnownIneq}) and the definition.
\end{proof}

\subsection{Measure of noncompactness $\lambda$.}

The measure of noncompactness $\lambda$ defined by Definition \ref{DefLambda} was studied in \cite{Keckic2}. Among others, the following properties of $\lambda$ were proved.

\begin{proposition}\label{osobine lambda}
The measure of noncompactness $\lambda$ has the following properties
\begin{enumerate}
    \item $\lambda(E)=\inf_{M\in \mathcal{F}}\sup_{x\in E} d(x,M)=\lim_{n\to\infty}\sup_{x\in E}\|x-P_nx\|$, where $\mathcal{F}$ is the set of all free finitely generated modules and $P_n\colon H_{\mA}\to H_{\mA}$ is given by $P_n(x_1,x_2,\ldots)=(x_1,x_2,\ldots,x_n,0,0,\ldots)$.
    \item if $E\subset F$, then $\lambda(E)\leq\lambda(F)$;
    \item $\lambda(E+F)\leq\lambda(E)+\lambda(F)$;
    \item $\lambda(B_1)=1$, where $B_1$ is the unit ball in $H_{\mA}$;
    \item $\lambda(E)\leq\sup_{x\in E}\|x\|$;
    \item $\lambda(E)=0$ iff $E$ is $\mA$-precompact.
\end{enumerate}
\end{proposition}

\subsection{Troitsky's topology}

\begin{definition}
Let $\mathcal N$ be a Hilbert $C^*$-module over $\mathcal{A}.$ A countable system $X=\{x_i\}$ of its elements is called admissible for a (possibly non-closed) submodule $\mathcal N^0\subset\mathcal N$ (or $\mathcal N^0$-admissible) if
\begin{itemize}
\item[1)] for every $x\in\mathcal N^0$ the series $\sum_i\skp{x}{x_i}\skp{x_i}{x}$ is convergent;
\item[2)] the sum in the previous item is bounded by $\skp{x}{x}$;
\item[3)] $\|x_i\|\leq1$ for each $i$.
\end{itemize}
\end{definition}

\begin{lemma} Let $\mathcal N$, $\mathcal N^0$ be as in the previous definition, and let $\Phi$ be a countable collection $\{\varphi_1,\varphi_2,\ldots\}$ of states on $\mathcal{A}$. Further, let $X=\{x_i\}$ be an $\mathcal N^0$-admissible system. The function $p_{X,\Phi}$ defined by
\begin{equation}\label{TroiSemiNorm}
p_{X,\Phi}(x)=\sqrt{\sup\limits_k\sum_{i=k}^{\infty}|\varphi_k(\skp{x}{x_i})|^2},\quad x\in N^0.
\end{equation}
is well-defined seminorm on $\mathcal N^0$ and $p_{X,\Phi}(x)\leq\|x\|$.
\end{lemma}

\begin{proof}
It was shown in \cite[Lemma 2.5, Theorem 2.6]{Troi}.
\end{proof}

\begin{definition}
Let $\mathcal N$ be a Hilbert $C^*-$module over a unital $C^*$-algebra $\mathcal{A}$, and let $\mathcal N^0$ be some (possibly non closed) submodule of $\mathcal N$. Troitsky's topology $\tau$ (or more informatively $(\mathcal N,\mathcal N_0)$-topology) is the locally convex topology defined by the family of seminorms (\ref{TroiSemiNorm}), where $X=\{x_i\}$ is $\mathcal N^0$-admissible and $\Phi=\{\varphi_1,\varphi_2,\ldots\}$ is a countable collection of states on $\mathcal{A}$.
\end{definition}

This family of seminorms, which determine the uniform structure on $\mathcal N$ was initially introduced in \cite{Troi}. This uniform structure perfectly fits the notion of $\mA$-compact operators.

\begin{theorem}\label{Troi teorema}\cite[Theorem 2.13]{Troi}
    Suppose, $F:\mathcal M \to\mathcal N$ is an adjointable morphism of Hilbert $C^*$-modules over $C^*$-algebra $\mA$. Then $F$ is $\mA$-compact if and only if $F(B)$ is $(\mathcal N,\mathcal N)$-totally bounded with respect to uniform structure $\tau$, where $B$ is the unit ball of $M$.
\end{theorem}

Although \cite{Troi} didn't deal with sets, but only with operators, the following statement is easy to derive.

\begin{corollary}\label{Troi posledica}
    Let $H_\mA$ be the standard Hilbert $C^*$-module over a unital $C^*$-algebra $\mA$ and let $P_n:H_\mA\to H_\mA$ be the projection given by (\ref{PeeN}). If $E$ is a norm bounded set, then $P_n(E)$ is totally bounded with respect to $\tau$.
\end{corollary}

\begin{proof}
    Let $e_n$, $n\ge 1$ be the standard basis for $H_\mA$ ($n$-th coordinate of $e_n$ being the unity of $\mA$ and otghers vanishes). Then $P_n(x)=\sum_{k=1}^n\skp x{e_k}e_k$, hence $\mA$ compact. If $E$ is norm bounded then $E\subseteq tB$, for a suitable $t>0$. Hence $P_n(E)\subseteq tP_n(B)$ wich is totally bounded.
\end{proof}

Finally, we quote a simple and well known property of positive elements.

\begin{lemma}\label{lema Marfi} \cite[Lemma 3.3.6]{Marfi} If $a$ is a normal element of a non-zero $C^*$-algebra $\mA$, then there is a state $\varphi$ of $\mA$ such that   $|\varphi(a)|=\|a\|$.
\end{lemma}

\section{Main results}

\subsection{Standard Hilbert module}

\begin{lemma}\label{UnitBall}
Let $H_{\mA}$ be the standard Hilbert module over $\mA$, and let $p_{\a},\a\in J$ be a family of seminorms such that $p_{\a}(x)\leq\|x\|$ for every $\a\in J$, $x\in H_{\mA}$, which turns $H_{\mA}$ into a locally convex space.

Let $B_1$ denote the unit ball in $H_\mA$. Then $\chi^*(B_1)\le 1$.
\end{lemma}

\begin{proof}
    Indeed, from $p_\a(x)\le\|x\|$ we have $B_1\subseteq \{x\in H_\mA\mid p_\a(x)<1\}$ for any $\a$.
\end{proof}

\begin{lemma}\label{sub}
Let $H_{\mA}$ be the standard Hilbert module over $\mA$, and let $p_{\a},\a\in J$ be a family of seminorms such that $p_{\a}(x)\leq\|x\|$ for every $\a\in J$, $x\in H_{\mA}$, which turns $H_{\mA}$ into a locally convex space.

Let $P_n:H_\mA\to H_\mA$ denote the projection (\ref{PeeN}) and let $E\subset H_{\mA}$ be a norm bounded set. If $P_nE$ is totally bounded (with respect to all $p_\alpha$) for every $n\in\mathbb{N}$, then
\begin{equation}\label{ChiLeLambda}
\mu^*(E)\leq\lambda(E)\mu^*(B_1),\qquad B_1\mbox{ unit ball in }H_\mA,
\end{equation}
for every measure of noncompactness $\mu$ that satisfies
\begin{equation}\label{TotalBound}
    \mu(F)=0\quad\mbox{iff}\quad F\mbox{ is totally bounded}.
\end{equation}

In particular, (\ref{ChiLeLambda}) holds for $\mu\in\{\chi,\alpha,I\}$ and the family of seminorms (\ref{TroiSemiNorm}).

Concerning Hausdorff measure of noncompactness, by virtue of Lemma \ref{UnitBall}, we have
\begin{equation}\label{HausdorffLeLambda}
\chi^*(E)\le\lambda(E).
\end{equation}
\end{lemma}

\begin{proof}
Let E be a bounded set. Since $E\subset P_nE+(I-P_n)E$, and since $\mu^*$ is subadditive (Proposition \ref{osobine MNC}), we have
$$\mu^*(E)\leq\mu^*(P_nE)+\mu^*((I-P_n)E).$$
Since, by assumption $\mu^*(P_nE)=0$ we have
$$\mu^*(E)\leq\mu^*((I-P_n)E)\leq\sup\limits_{x\in E}\|(I-P_n)x\|\mu^*(B_1),$$
for all $n\in\mathbb{N}$. (The last inequality follows from $(I-P_n)(E)\subseteq\sup_{x\in E}\|(I-P_n)x\|B_1$.) Therefore, by Proposition \ref{osobine lambda}, $\mu^*(E)\leq\lambda(E)\mu^*(B_1)$.

By Corollary \ref{Troi posledica} $P_nE$ is totally bounded, with respect to seminorms (\ref{TroiSemiNorm}), whereas by Proposition \ref{osobine MNC}, Theorem \ref{IvanAxiome} and Remark \ref{DrugiSmer} all $\chi^*$, $\alpha^*$ and $I^*$ are subadditive and satisfy (\ref{TotalBound}).
\end{proof}

\begin{lemma}\label{Lema1}
Let $H_{\mathcal{A}}$ be a standard Hilbert $C^*$-module over $\mathcal{A}$ and $E\subset H_{\mathcal{A}}$. Let $Y=\{y_i\}, y_i\in E$ be a countable system and $(n_k)$ be an increasing sequence in $\mathbb{N}$. Then $$X=\{z_i\},\quad z_i=\|P_{n_i}(I-P_{n_{i-1}})y_i\|^{-1}P_{n_i}(I-P_{n_{i-1}})y_i$$
is $H_{\mathcal{A}}$-admissible, where $P_n$ is defined by (\ref{PeeN}).
\end{lemma}

\begin{proof}
     1) Let $x=(x_1,x_2,\ldots)\in H_{\mathcal{A}}$ and $\varepsilon>0$ be arbitrary. Since $\sum_{i=1}^{\infty}x_i^*x_i$ is convergent in norm there exists $m_0\in\mathbb{N}$ such that for any $m>m_0$ and $p_1\in\mathbb{N}$ we have
    $$\left\|\sum_{i=m}^{m+p_1}x_i^*x_i\right\|<\varepsilon.$$
Denote $T_i=P_{n_i}(I-P_{n_{i-1}})$; then $z_i=\|T_iy_i\|^{-1}T_iy_i$ and hence, taking into account that $T_i$ are selfadjoint projections:
\begin{align*}
    \left\|\sum_{i=m}^{m+p}\skp{x}{z_i}\skp{z_i}{x}\right\|&=
     \left\|\sum_{i=m}^{m+p}\skp{x}{\|T_iy_i\|^{-1}T_iy_i}\skp{\|T_iy_i\|^{-1}T_iy_i}{x}\right\|\\
     &=\left\|\sum_{i=m}^{m+p}\skp{T_ix}{\|T_iy_i\|^{-1}T_iy_i}\skp{\|T_iy_i\|^{-1}T_iy_i}{T_ix}\right\|
\end{align*}
However, by Cauchy Scwartz inequality \cite[Proposition 1.2.4]{Manuilov},
$$\skp{T_ix}{\|T_iy_i\|^{-1}T_iy_i}\skp{\|T_iy_i\|^{-1}T_iy_i}{T_ix}\le\skp{T_ix}{T_ix}.$$
Hence
\begin{align*}
    \left\|\sum_{i=m}^{m+p}\skp{x}{z_i}\skp{z_i}{x}\right\|&\le\left\|\sum_{i=m}^{m+p}\skp{T_ix}{T_ix}\right\|=\\
    &=\left\|\sum_{i=n_{m-1}+1}^{n_{m+p}}x_i^*x_i\right\|=
     \left\|\sum_{i=n_{m-1}+1}^{n_{m-1}+p_1}x_i^*x_i\right\|,
\end{align*}
for $p_1=n_{m+p}-n_{m-1}-1$. Since $n_{m-1}+1\geq m$, the last term does not exceed $\varepsilon$. Therefore, $\sum_{i=1}^{\infty}\skp{x}{z_i}\skp{z_i}{x}$ is convergent in the norm.

    2) For every $x\in H_{\mathcal{A}}$
    \begin{align*}
        \sum\limits_{i}\skp{x}{z_i}\skp{z_i}{x}\leq\sum_i\skp{T_ix}{T_ix}=\skp{x}{x}
    \end{align*}

    3) Obviously, $\|z_i\|\leq1$ for each $i$.
\end{proof}

\begin{theorem}\label{thm2}
Let $E\subset H_{\mA}$ be a bounded set, let $u\in A$ be a unitary and let $\mu$ stands for any of Hausoorff, Kuratowski or Istr\u{a}\c{t}escu measure of noncompactness. Then $\mu^*(Eu)=\mu^*(E)$.
\end{theorem}

\begin{proof}
Let $\varepsilon>0$ be arbitrary. There exists an $H_{\mA}$-admissible pair $(X,\Phi), X=\{x_i\},\Phi=\{\varphi_i\}$ such that  $\mu(Eu)(p_{\Phi,X})>\mu^*(Eu)-\varepsilon.$ Let $\varphi_i^u(x)=\varphi_i(u^*xu)$  and $x_i^u=x_iu^*$ for $i\in\mathbb{N}$.  Obviously, $\varphi_i^u(1)=1$ and $\varphi_i^u(x)\geq0$ whenever $x\geq0$, so $\varphi_i^u$ are states on $H_{\mA}$. Since
$$\skp{x}{x_i^u}\skp{x_i^u}{x}=\skp{x}{x_i}u^*u\skp{x_i}{x}=\skp{x}{x_i}\skp{x_i}{x},$$
and
$$\|x_i^u\|=\|x_iu^*\|\leq\|x_i\|\,\cdot\|u^*\|=\|x_i\|,$$
the pair $(X^u,\Phi^u)$ is $H_{\mA}$-admissible, where $\Phi^u=\{\varphi_1^u,\varphi_2^u,\ldots\}$ and $X^u=\{x_1^u,x_2^u,\ldots\}$.

Thus
\begin{eqnarray*}
p_{X,\Phi}^2(xu)&=&\sup\limits_{k}\sum\limits_{i=k}^{+\infty}|\varphi_k(\skp{xu}{x_i})|^2=\sup\limits_{k}\sum\limits_{i=k}^{+\infty}|\varphi_k(u^*\skp{x}{x_i})|^2\\
&=&\sup\limits_{k}\sum\limits_{i=k}^{+\infty}|\varphi_k(u^*\skp{x}{x_iu^*}u)|^2=
\sup\limits_{k}\sum\limits_{i=k}^{+\infty}|\varphi_k^u(\skp{x}{x_iu^*})|^2\\
&=&\sup\limits_{k}\sum\limits_{i=k}^{+\infty}|\varphi_k^u(\skp{x}{x_i^u})|^2=p_{X^u,\Phi^u}^2(x).
\end{eqnarray*}
Therefore, $[\mu(E)](p_{X^u,\Phi^u})>\mu^*(Eu)-\varepsilon$ and hence $\mu^*(E)\geq\mu^*(Eu)$. The opposite inequality follows by $E=(Eu)u^{-1}.$
\end{proof}

\begin{theorem}\label{thm3}
For any bounded set $E\subset H_{\mA}$, we have
$$\chi^*(E)=\lambda(E).$$
\end{theorem}

\begin{proof}
By Lemma \ref{sub} we have $\chi^*(E)\le\lambda(E)$. Therefore it is enough to prove the opposite inequality.

Let $\lambda=\lambda(E)=\inf_{i\geq1}\sup_{x\in E}\|(I-P_i)x\|$ and let $0<\varepsilon<\lambda$ be arbitrary. Suppose that the Hausdorff measure of noncompactness $\chi^*(E)$ is at most $\lambda-\varepsilon$, that is for every pair $(X,\Phi)$ with $H_{\mA}$-admissible $X$ there are $y_1,y_2,\ldots,y_l\in H_{\mA}$ such that for every $y\in E$ there exists $k\in\{1,2,\ldots,l\}$  such that $d_{X,\Phi}(y,y_k)<\lambda-\varepsilon$.

Since $\sup_{x\in E}\|(I-P_i)x\|=\lambda$, there is a sequence $z_i\in E$ such that $\|(I-P_i)z_i\|>\lambda-\varepsilon/4$. Further, there is an increasing sequence $i_j$ of positive integers for which $\|P_{i_{j+1}}z_{i_j}-z_{i_j}\|<\varepsilon/4$. We have
\begin{eqnarray*}
\|(P_{i_{j+1}}-P_{i_j})z_{i_j}\|&=&\|(P_{i_j}z_{i_j}-z_{i_j})-(P_{i_{j+1}}z_{i_j}-z_{i_j})\|\\
&>&\|(I-P_{i_j})z_{i_j}\|-\|P_{i_{j+1}}z_{i_j}-z_{i_j}\|>\lambda-\varepsilon/4-\varepsilon/4\\
&=&\lambda-\varepsilon/2.
\end{eqnarray*}

Choose $x_j=\|(P_{i_{j+1}}-P_{i_j})z_{i_j}\|^{-1}(P_{i_{j+1}}-P_{i_j})z_{i_j})$. By Lemma \ref{Lema1} $X_0=\{x_1,x_2,\dots\}$ is $H_\mA$-admissible.

Note that
$$\skp{x_j}{(P_{i_{j+1}}-P_{i_j})z_{i_j}}=\|(P_{i_{j+1}}-P_{i_j})z_{i_j}\|\skp{x_j}{x_j}\ge0,$$
as well as
$$\|\skp{x_j}{(P_{i_{j+1}}-P_{i_j})z_{i_j}}\|=\|(P_{i_{j+1}}-P_{i_j})z_{i_j}\|>\lambda-\varepsilon/2.$$
Therefore, by Lemma \ref{lema Marfi} there are states $\varphi_j$ such that $|\varphi_j(\skp{x_j}{(P_{i_{j+1}}-P_{i_j})z_{i_j}})|>\lambda-\varepsilon/2$.

For the seminorm $p_{X_0,\Phi_0}$, $\Phi_0=\{\va_1,\va_2,\ldots\}$, let $y_1,y_2,\ldots,y_l\in H_{\mA}$ be the corresponding $\lambda-\varepsilon$ net. One can find a number $j_0$ such that
\begin{eqnarray*}
\|(I-P_{i_j})y_k\|<\varepsilon/4\quad j\geq j_0,\quad k=1,\ldots,D
\end{eqnarray*}
and hence, for $j\geq j_0,\,k=1,2,\ldots,l$
\begin{eqnarray*}
\|\skp{y_k}{x_j}\|
&=&\|\skp{\|(P_{i_{j+1}}-P_{i_j})z_{i_j}\|^{-1}(P_{i_{j+1}}-P_{i_j})z_{i_j}}{(P_{i_{j+1}}-P_{i_j})y_k}\|\\
&\leq&\|(P_{i_{j+1}}-P_{i_j})y_k\|<\varepsilon/2.
\end{eqnarray*}
Then, for all $k=1,\ldots,l$ and $y=z_{i_j}$,
\begin{eqnarray*}
d_{X_0,\Phi_0}(y,y_k)&=&\sqrt{\sup\limits_m\sum_{i=m}^{\infty}|\varphi_k(\skp{y-y_k}{x_i})|^2}\\
&\geq&|\va_j(\skp{z_{i_j}-y_k}{x_j})|\geq|\va_j(\skp{z_{i_j}}{x_j})|-|\va_j(\skp{y_k}{x_j})|\\
&=&|\va_j(\skp{(P_{i_{j+1}}-P_{i_j})z_{i_j}}{x_j})|-|\va_j(\skp{y_k}{x_j})|>\lambda-\frac{\varepsilon}{2}-\frac{\varepsilon}{2}\\
&=&\lambda-\varepsilon.
\end{eqnarray*}
This contradicts the choice of $y_1,\ldots,y_l$. Hence, the assumption that there is a $\lambda-\varepsilon$ net can not hold. Thus, $\lambda(E)\leq\chi^*(E)$.
\end{proof}

\begin{remark}
    Following \cite[Theorem 4.7]{Keckic2}, one might prove the inequality $\lambda(E)\le\sqrt{\|E\|I^*(E)}$. However, due to the previous theorem, it is trivial. Indeed, $\lambda(E)=\chi^*(E)\le\|E\|$, $I^*(E)$.
\end{remark}

\begin{corollary}
Let $E$ be a bounded set in standard Hilbert module $H_{\mA}$. There holds
$$\chi^*(E)=\lambda(E)\leq I^*(E)\leq\alpha^*(E)\leq2\chi^*(E)=2\lambda(E).$$
\end{corollary}

\begin{proof}
Follows immediately from Lemma \ref{MNCinequalities} and Theorem \ref{thm3}.
\end{proof}

\begin{corollary}
    Let $B_1$ be the closed unit ball of $H_{\mA}$. Then $\chi(B_1)=1.$
\end{corollary}

\begin{proof}
Follows immediately from Proposition \ref{osobine lambda} and Theorem \ref{thm3}.
\end{proof}

\subsection{Arbitrary countably generated Hilbert module}

We want to extend our results from standard Hilbert module $H_\mA$ to an arbitrary countably generated Hilbert module $\mathcal M$ over $\mA$. As it is usual, we use the Kasparov stabilization theorem $\mathcal M\oplus H_\mA\cong H_\mA$.

Note that the measure of noncompactness $\lambda$ is defined via free finitely generated submodules which was suitable for the standard Hilbert module. For more general purpose we need to calculate $\lambda$ by finitely generated projective submodules.

\begin{lemma}
    Let $E$ be a bounded subset of $H_\mA$. Then
    $$\lambda(E)=\lim_{n\to\infty}\sup_{x\in E}\|x-P_nx\|=\sup_{M\in\mathcal P}d(E,M)=\sup_{M\in\mathcal P}\sup_{x\in E}d(x,M),
    $$
    where $P_n$ are defined by (\ref{PeeN}), and $\mathcal P$ denotes the set of all finitely generated projective submodules of $H_\mA$.
\end{lemma}

\begin{proof}
    The first equality was proved in \cite[Proposition 2.2]{Keckic2} and literally the same argument leads to the second. Since it is short we repeat it.

    Denote
    $$\hat\lambda(E)=\sup_{M\in\mathcal P}d(E,M).$$
    Since $P_nH_\mA$ is finitely generated and projective (and moreover free), we have $\hat\lambda(E)\le\sup_{x\in E}\|x-P_nx\|$ and hence $\hat\lambda(E)\le\lim_{n\to\infty}\sup_{x\in E}\|x-P_nx\|$.

    For the opposite inequality, let $\varepsilon>0$ be arbitrary. Then, there is a finitely generated projective module $L\subseteq H_\mA$ such that $\sup_{x\in E}d(x,L)<\hat\lambda(E)+\varepsilon$. By \cite[Theorem 15.4.2]{Wegge}, the corresponding projection $P_L$ is $\mA$-compact, and by \cite[Proposition 2.2.1]{Manuilov} $\|P_L(I-P_n)\|\to0$, as $n\to\infty$. Then
    \begin{multline*}\|x-P_nx\|=d(x,P_nH_\mA)\le\|x-P_nP_Lx\|\le\\\le\|x-P_Lx\|+\|P_Lx-P_nP_Lx\|\le\hat\lambda(E)+\varepsilon+\|E\|\,\|P_L-P_nP_L||,
    \end{multline*}
    and taking $\sup_{x\in E}$ and limit we obtain the desired inequality.
\end{proof}

We will use the previous characterization as a definition of measure of noncompactness $\lambda$ on an arbitrary countably generated Hilbert module.

\begin{definition}\label{DefLambda2}
    Let $\mathcal M$ be a countably generated Hilbert module over a $C^*$-algebra $\mA$, and let $E\subseteq\mathcal M$ be a bounded set. We define
    $$\lambda(E)=\inf_{L\in\mathcal P}\sup_{x\in E}d(x,L),$$
    where $\mathcal P$ denotes the set of all finitely generated projective submodules of $\mathcal M$.
\end{definition}

\begin{proposition}\label{LambdaGreaterModule}
    Let $\mathcal M_1\subseteq\mathcal M_2$ be two countably generated Hilbert modules over the same $C^*$-algebra $\mA$, such that $\mathcal M_1$ is complemented in $\mathcal M_2$, i.e.
    \begin{equation}\label{M_1compM_2}
    \mathcal M_2=\mathcal M_1\oplus\mathcal N
    \end{equation}
    for some $\mathcal N$ and let $E\subseteq \mathcal M_1$ be a bounded set. Consider two functions $\lambda_j(E)$, $j=1,2$, the measures of noncompactness regarding $E$ as a subset of $\mathcal M_j$. Then $\lambda_1(E)=\lambda_2(E)$.
\end{proposition}

\begin{proof}
    Obviously, $\lambda_1(E)\ge\lambda_2(E)$ since in the latter the infimum is taken over a larger set. Conversely, let $\varepsilon>0$ be arbitrary. Then, there is a finitely generated projective $L\subseteq\mathcal M_2$ such that $\sup_{x\in E} d(x,L)<\lambda_2(E)+\varepsilon$.

    By (\ref{M_1compM_2}) there exists the projection $Q:\mathcal M_2\to\mathcal M_1$. Obviously, $Q(L)\subseteq\mathcal M_1$ is finitely generated. It is, also, projective, for if $\mA^n\cong L\oplus L_1$ for some $L_1$, then $\mA^n\cong Q(L)\oplus(I-Q)(L)\oplus L_1$.

    Now, let $x\in E\subseteq\mathcal M_1$ and $y\in L$. Then $x=Qx$ and therefore
    \begin{multline*}
    \skp{x-y}{x-y}=\skp{Q(x-y)-(I-Q)y}{Q(x-y)-(I-Q)y}=\\=\skp{Q(x-y)}{Q(x-y)}+\skp{(I-Q)y}{(I-Q)y}\ge\skp{Q(x-y)}{Q(x-y)}.
    \end{multline*}
    Hence
    $$\|x-y\|\ge\|Qx-Qy\|=\|x-Qy\|\ge d(x,Q(L)).$$
    Taking the infimum over all $y\in L$ we obtain $d(x,L)\ge d(x,Q(L))$, and taking the supremum over $x\in E$ we get
    $$\sup_{x\in E}d(x,L)\ge\sup_{x\in E}d(x,Q(L))\ge\lambda_1(E),$$
    which finally leads to $\lambda_2(E)\ge\lambda_1(E)$.
\end{proof}

\begin{proposition}\label{ChiGreaterModule}
    Let $\mathcal M_j$ $j=1,2$ be two countably generated modules and let $M_j^0$ be their submodules. Also, let $\chi$ be Hausdorff measure of noncompactness of Troitsky $(\mathcal M,\mathcal M^0)$ topology, where $\mathcal M=\mathcal M_1\oplus\mathcal M_2$ and $\mathcal M^0=\mathcal M_1^0\oplus\mathcal M_2^0$. Also, let $\chi_j$ be the corresponding Hausdorff measures of noncompactness related to $(\mathcal M_j;\mathcal M_j^0)$ topology.

    Denote by $p_j$ projection from $\mathcal M$ onto $\mathcal M_j$. If $E\subseteq\mathcal M$ is a bounded set, then
    \begin{equation}\label{ChiOplus}
    \max\{\chi_1(p_1E),\chi_2(p_2E)\}\le\chi(E)\le\chi_1(p_1E)+\chi_2(p_2E).
    \end{equation}
\end{proposition}

\begin{proof}
    The proof can be obtained by a slight modification of the proof of \cite[Lemma 2.15]{Troi}.

    Denote $J_j=p_j^*:\mathcal M_j\to\mathcal M$ the corresponding inclusions. Let $X=\{x_i\}$ be an arbitrary sequence in $\mathcal M$ admissible for $\mathcal M^0$ and let $\Phi=\{\varphi_i\}$ be an arbitrary sequence of states on $\mA$. It is easy to see that $X^{(j)}=\{p_jx_i\}$ is admissible for $\mathcal M_j^0$, for details see the proof of \cite[Lemma 2.15]{Troi}, as well as (for $j=1,2$)
    $$p_{X,\Phi}(J_jy)=\sqrt{\sup_{k\ge1}\sum_{i=k}^\infty|\varphi_k(\skp{J_jy}{x_i})|^2}=\sqrt{\sup_{k\ge1}\sum_{i=k}^\infty|\varphi_k(\skp{y}{p_jx_i})|^2}=p_{X^{(j)},\Phi}(y).$$

    Let $\delta_1>\chi_1(p_1E)$, $\delta_2>\chi_2(p_2E)$ be arbitrary. Then, there is a finite $\delta_1$ net for $p_1E$ in the seminorm $p_{X^{(1)}},\Phi$, say $z_1,\dots,z_m\in\mathcal M_1$ and a finite $\delta_2$ net for $p_2E$ in the seminorm $p_{X^{(2)},\Phi}$, say $w_1,\dots,w_r\in\mathcal M_2$.

    We claim that $J_1z_i+J_2w_s$, $1\le i\le m$, $1\le s\le r$ make a finite $\delta_1+\delta_2$ net for $E$. Indeed, if $y\in E$, then $y=J_1p_1y+J_2p_2y$. Also, there is an $1\le i\le k$ such that $p_{X^{(1)},\Phi}(p_1y-z_i)<\delta_1$ as well as an $1\le s\le r$ such that $p_{X^{(2)},\Phi}(p_2y-w_s)<\delta_2$. Then
    \begin{multline*}
    p_{X,\Phi}(y-J_1z_k-J_2w_s)\le p_{X,\Phi}(J_1(p_1y-z_k))+p_{X,\Phi}(J_2(p_2y-w_s))=\\
    =p_{X^{(1)},\Phi}(p_1y-z_k)+p_{X^{(2)},\Phi}(p_2y-w_s)\le\delta_1+\delta_2.
    \end{multline*}

    Thus, $\chi(E)$ does not exceed $\delta_1+\delta_2$, and the second inequality in (\ref{ChiOplus}) is proved.

    For the first inequality, let $\Phi=\{\varphi_i\}$ and let $X=\{x_i\}$ be an arbitray sequence in $\mathcal M_1$ admissible for $\mathcal M_1^0$. Then it is easy to see that $J_1X=\{J_1x_i\}$ is admissible for $\mathcal M^0$ as well as
    $$p_{X,\Phi}(p_1y)=\sqrt{\sup_{k\ge1}\sum_{i=k}^\infty|\varphi_k(\skp{p_1y}{x_i})|^2}=\sqrt{\sup_{k\ge1}\sum_{i=k}^\infty|\varphi_k(\skp{y}{J_1x_i})|^2}=p_{J_1X,\Phi}(y).$$

    If $\delta>\chi(E)$ is arbitrary, then there is a finite $\delta$ net in the seminorm $p_{J_1X,\Phi}$ for $E$, say $z_1,\dots,z_k$. Consequently, for an arbitrary $y\in p_1E$, there is an $1\le i\le k$ such that $p_{J_1X,\Phi}(J_1y-z_i)<\delta$. Then
    $$p_{X,\Phi}(y-p_1z_i)=p_{X,\Phi}(p_1(J_1y-z_i))=p_{J_1X,\Phi}(J_1y-z_i)<\delta.$$
    Thus $\chi_1(E)\le\delta$, which proves the first inequality.
\end{proof}

\begin{corollary}\label{ChiCorollary}
    Let $\mathcal M_1\subseteq\mathcal M_2$ be two countably generated Hilbert modules over $\mA$, such that $\mathcal M_1$ is complemented in $\mathcal M_2$, and let $E\subseteq\mathcal M_1$ be an arbitrary bounded set.

    Then $\chi^*_1(E)=\chi_2^*(E)$, where $\chi_j^*$ denote the Hausdorff measure of noncompactnes for Troitsky's topology with respect to $\mathcal M_j$.
\end{corollary}

\begin{proof}
    After some changes in notations we can apply the previous proposition with $\mathcal M_j^0=\mathcal M_j$ and $E\subseteq\mathcal M_1$. Then $p_2E$ is trivial, implying $\chi_2^*(p_2E)=0$. Thus the equation (\ref{ChiOplus}) becomes
    $$\chi_1^*(E)\le\chi^*(E)\le\chi_1^*(E)+0.$$
\end{proof}

\begin{theorem}
    Let $\mathcal M$ be a countably generated Hilbert $C^*$-module over a $C^*$-algebra $\mA$. Let $\chi$ denote the Hausdorff measure of noncompactness with respect to Troitsky's topology $\tau$, and let $\lambda$ denote the $\mA$ measure of noncompactness defined in Definition \ref{DefLambda2}.

    If $E\subseteq\mathcal M$ is a bounded set, then $\chi(E)=\lambda(E).$
\end{theorem}

\begin{proof}
    By Kasparov stabilization theorem, we have
    $$\mathcal M\oplus H_{\mA}\cong H_{\mA}=:\mathcal M_2.$$

    Thus $\mathcal M$ is complemented in $M_2$ and $M_2\cong H_{\mA}$. Denote by $\chi^*_2(E)$ and $\lambda_2(E)$ the corresponding measures of noncompactness with respect to $\mathcal M_2\cong H_{\mA}$. By Theorem \ref{thm3} we have $\lambda_2(E)=\chi^*_2(E)$.

    On the other hand, by Proposition \ref{LambdaGreaterModule} we have $\lambda(E)=\lambda_2(E)$. Finally, apply Corollary \ref{ChiCorollary} to obtain $\chi^*(E)=\chi_2^*(E)$.
\end{proof}

\subsection{Measures of noncompactness of operators}

\begin{definition}
Let $\mathcal M$ be a countably generated Hilbert module over a $C^*$-algebra $\mA$ and let $T\in B^a(\mathcal M$ be an adjointable operator. The functions $\alpha_o^*,\chi_o^*,I_o^*,\lambda_o^*: B^a(\mathcal M)\to[0,+\infty)$ defined by
$$\chi_o^*(T)=\chi^*(T(B_1)),\quad \alpha_o^*(T)=\alpha^*(T(B_1)),$$
$$I_o^*(T)=I^*(T(B_1)),\quad \lambda_o(T)=\lambda(T(B_1)),$$
where $B_1$ is the closed unit ball in $\mathcal M$ are called, respectively, Hausdorff, Kuratowski, Istr\u{a}\c{t}escu and $\mA$ measure of noncompactness of the operator $T$.
\end{definition}

\begin{proposition}
Let $\mathcal M$ be a countably generated Hilbert module over a $C^*$-algebra $\mA$ and let $T\in B^a(\mathcal M)$. Then
$$\chi_0^*(T)=\inf\{k\geq0\mid\chi^*(T(E))\leq k\chi^*(M)\,\,\mbox{for each bounded set } E\}.$$
\end{proposition}

\begin{proof}
The proof is, essentially, the same as in the case of metric spaces (see \cite[Theorem 2.25]{Rakocevic}.)

Indeed, let $k>0$ be such that
\begin{equation}\label{ChiOgranicenje}
    \chi^*(T(E))\le k\chi^*(E),
\end{equation}
holds for all bounded $E\subseteq\mathcal M$. For $E=B_1$ we obtain $\chi_0(T)=\chi^*(T(B_1))\le k\chi(B_1)=k$. Thus the infimum of such $k$ is greater or equal to $\chi_0(T)$.

For the other inequality, let $E$ be an arbitrary bounded set, let $k>\chi_0(T)$, and let $\eta>\chi^*(E)$. For any seminorm $p_\alpha$ there are two finite nets, the first $y_1$, $y_2$, $\dots$, $y_l$ makes a finite $k$ net for $B_1$, and the other $z_1$, $z_2$, $\dots$, $z_p$ makes a finite $\eta$ net for $E$. If $B(a;\delta)$ denote the ball with center at $a$ and radius $\delta$, then $E\subseteq\cup_{i=1}^pB(z_i;\eta)$ and hence
$$T(E)\subseteq\bigcup_{i=1}^pT(B(z_i;\eta))=\bigcup_{i=1}^p(Tz_i+\eta T(B_1)).$$
On the other hand, $T(B_1)\subseteq\cup_{j=1}^lB(y_j;k)$ and therefore
$$T(E)\subseteq\bigcup_{i=1}^p\left(Tz_i+\eta\bigcup_{j=1}^lB(y_j;k)\right)=\bigcup_{i=1}^p\bigcup_{j=1}^lB(Tz_i+\eta y_j;\eta k).$$
Thus $\chi^*(T(E))$ does not exceed $\eta k$, wich can be arbitrarily close to $k\chi^*(E)$.
\end{proof}

\begin{proposition}
Let $\mathcal{A}$ be an arbitrary $C^*$-algebra, let $T,S\in B^a(H_{\mathcal{A}})$ and let $\mu$ stands for any of MNCs $\alpha,\chi,I$. Then
\begin{itemize}
    \item[a)] All $\alpha_o^*,\chi_o^*$ and $I_o^*$ are subadditive and positive homogeneous, i.e. there holds
    $$\mu_o^*(T+S)\leq\mu_o^*(T)+\mu_o^*(S),\quad\mu_o(cT)=c\mu^*(T)\quad\mbox{for all } c>0.$$
    \item[b)] The functions $\alpha_o^*,\chi_o^*, I_o^*$ and $\lambda_o$ are equivalent to each other, that is,
    $$\chi_o^*(T)\leq I_o(T)\leq\alpha_o^*(T)\leq2\chi_o^*(T),\quad \chi_o^*(T)=\lambda_o(T).$$
    \item[c)] $\chi_o^*(T)\leq\|T\|$ and $\alpha_o^*(T),I_o^*(T)\leq2\|T\|$.
    \item[d)] Operator $T$ is $\mA$-compact iff $\lambda_0(T)=0$ iff $\mu_o^*(T)=0$.
    \item[e)] $\mu_o^*(T+K)=\mu_o^*(T)$, as well as $\lambda_o(T+K)=\lambda_o(T)$ for all $\mA$-compact operators $K$.
    \item[f)] There holds
$$\chi_o^*(T)=\lambda_o(T)\leq I_o^*(T)\leq\alpha_o^*(T)\leq 2\chi_o^*(T).$$
    \end{itemize}
\end{proposition}

\begin{proof}
Parts a), b), c), and f)  follow from the properties of  $\alpha,\chi, I$ and $\lambda$.
Part d) follows from part b) and Theorem \ref{Troi teorema}.
Parts e) follows from part d) and the properties of  $\alpha,\chi, I$ and $\lambda$.
\end{proof}

\end{document}